\documentclass{amsart}
\usepackage{amscd,amsmath,amsthm,amssymb}

\usepackage{picture}

\usepackage{pstcol,pst-plot,pst-3d}
\usepackage{lmodern,pst-node}
\usepackage{multicol}
\usepackage{epic,eepic}
\usepackage{amsfonts,amssymb,amscd,amsmath,enumerate,verbatim}
\psset{unit=0.7cm,linewidth=0.8pt,arrowsize=2.5pt 4}

\newpsstyle{fatline}{linewidth=1.5pt}
\newpsstyle{fyp}{fillstyle=solid,fillcolor=verylight}
\definecolor{verylight}{gray}{0.97}
\definecolor{light}{gray}{0.9}
\definecolor{medium}{gray}{0.85}
\definecolor{dark}{gray}{0.6}
\def\frk{\frak}               

\def\Phi{{\frk n}}
\def\Phi{{\frk N}}
\def\bb{{\frk b}}
\def\cc{{\frk c}}
\def\MI{{\mathcal I}}

\def\Ic{{\mathcal I}}

\def\Dc{{\mathcal D}}
\def\bb{{\bold b}}

\def\ub{{\bold u}}
\def\opn#1#2{\def#1{\operatorname{#2}}} 
\opn\chara{char} \opn\length{\ell} \opn\pd{pd} \opn\rk{rk}
\opn\projdim{proj\,dim} \opn\injdim{inj\,dim} \opn\rank{rank}
\opn\depth{depth} \opn\grade{grade} \opn\height{height}
\opn\embdim{emb\,dim} \opn\codim{codim}

\opn\Tr{Tr} \opn\bigrank{big\,rank}
\opn\superheight{superheight}\opn\lcm{lcm}
\opn\trdeg{tr\,deg}
\opn\reg{reg} \opn\lreg{lreg} \opn\ini{in} \opn\lpd{lpd}
\opn\size{size} \opn\sdepth{sdepth}
\opn\link{link}\opn\fdepth{fdepth}\opn\lex{lex}
\opn\div{div} \opn\Div{Div} \opn\cl{cl} \opn\Cl{Cl}
\opn\Spec{Spec} \opn\Supp{Supp} \opn\supp{supp} \opn\Sing{Sing}
\opn\Ass{Ass} \opn\Min{Min}\opn\Mon{Mon}
\opn\Ann{Ann} \opn\Rad{Rad} \opn\Soc{Soc}
\opn\Im{Im} \opn\Ker{Ker} \opn\Coker{Coker} \opn\Am{Am}
\opn\Hom{Hom} \opn\Tor{Tor} \opn\Ext{Ext} \opn\End{End}
\opn\Aut{Aut} \opn\id{id}

\opn\nat{nat}
\opn\pff{pf}
\opn\Pf{Pf} \opn\GL{GL} \opn\SL{SL} \opn\mod{mod} \opn\ord{ord}
\opn\Gin{Gin} \opn\Hilb{Hilb}\opn\sort{sort}
\opn\ini{in}\opn{\rev}{rev}
\opn\aff{aff} \opn\con{conv} \opn\relint{relint} \opn\st{st}
\opn\lk{lk} \opn\cn{cn} \opn\core{core} \opn\vol{vol}
\opn\link{link} \opn\star{star}\opn\lex{lex}\opn\set{set}
\opn\gr{gr}

\def\pot#1#2{#1[\kern-0.28ex[#2]\kern-0.28ex]}

\opn\dirlim{\underrightarrow{\lim}}
\opn\inivlim{\underleftarrow{\lim}}
\let\to=\rightarrow

\def\Implies{\ifmmode\Longrightarrow \else
        \unskip${}\Longrightarrow{}$\ignorespaces\fi}
\def\implies{\ifmmode\Rightarrow \else
        \unskip${}\Rightarrow{}$\ignorespaces\fi}
\def\iff{\ifmmode\Longleftrightarrow \else
        \unskip${}\Longleftrightarrow{}$\ignorespaces\fi}

\let\:=\colon
\newtheorem{Theorem}{Theorem}
\newtheorem{Lemma}[Theorem]{Lemma}
\newtheorem{Corollary}[Theorem]{Corollary}
\newtheorem{Proposition}[Theorem]{Proposition}

\newtheorem{Example}[Theorem]{Example}

\newtheorem{Definition}[Theorem]{Definition}

\let\epsilon\varepsilon
\let\kappa=\varkappa
\def\qed{\ifhmode\textqed\fi
      \ifmmode\ifinner\quad\qedsymbol\else\dispqed\fi\fi}
\def\textqed{\unskip\nobreak\penalty50
       \hskip2em\hbox{}\nobreak\hfil\qedsymbol
       \parfillskip=0pt \finalhyphendemerits=0}
\def\dispqed{\rlap{\qquad\qedsymbol}}

\opn\dis{dis}
\def\pnt{{\raise0.5mm\hbox{\large\bf.}}}

\opn\Lex{Lex}

\begin{document}

\title {Regularity of joint-meet ideals of distributive lattices}

\author {Viviana Ene,  Ayesha Asloob Qureshi, Asia Rauf}

\address{Viviana Ene, Faculty of Mathematics and Computer Science, Ovidius University, Bd.\ Mamaia 124,
 900527 Constanta, Romania, and
 \newline
 \indent Simion Stoilow Institute of Mathematics of the Romanian Academy, Research group of the project  ID-PCE-2011-1023,
 P.O.Box 1-764, Bucharest 014700, Romania} \email{vivian@univ-ovidius.ro}

\address{Ayesha Asloob Qureshi, The Abdus Salam International Center of Theoretical Physics, Trieste, Italy} \email{ayesqi@gmail.com}

\address{Asia Rauf, Department of Mathematics, Lahore Leads University, 5 Tipu Block New Garden Town, Lahore 54000, Pakistan} \email{asia.rauf@gmail.com}

\begin{abstract} Let $L$ be a distributive lattice and $R(L)$ the associated Hibi ring.
We compute $\reg R(L)$ when $L$ is  a planar  lattice and give  bounds for $\reg R(L)$ when $L$ is non-planar, in terms of the combinatorial data of $L.$  As a consequence, we characterize the distributive lattices $L$ for which the associated Hibi ring has a linear resolution.
\end{abstract}

\thanks{The first author was supported by the grant UEFISCDI,  PN-II-ID-PCE- 2011-3-1023.}
\maketitle

\section*{Introduction}

Let $L$ be a finite distributive lattice and $K[L]$  the polynomial ring over a field $K$. The {\em join-meet} or {\em Hibi ideal} of $L$, denoted $I_L,$ is generated by all the  binomials 
$f_{ab}= ab- (a \vee b)(a \wedge b)$ where  $a,b \in L$ are incomparable. The {\em Hibi ring} of $L$ is  $R(L)=K[L]/ I_L$. $R(L)$ is a 
Cohen-Macaulay normal domain  as it was shown in \cite{H}. Its properties were investigated in  \cite{H},  \cite{H2}, \cite{H3}.  The Gr\"obner bases of $I_L$ with respect to 
various monomial orders have been studied; see, for instance, \cite{AHH}, \cite{HH1}, \cite{H}, \cite{Q}. 

Our aim is to study the regularity of $I_L$ for a distributive lattice $L$. When $L$ is a planar lattice, we give the regularity formula  in Theorem~\ref{planar} in terms of the combinatorics of the lattice. For non-planar lattices, we show in Theorem~\ref{non-planar}  that $\reg R(L)$ is greater than or equal to the maximal number of pairwise incomparable join-irreducible elements minus $1$. These two results enable us to derive that $I_L$ has a $2$-linear resolution if and only if  $L$ is the divisor lattice of $2\cdot 3^a$ for some $a \geq 1$; see Corollary~\ref{linres}. For other nice properties of this lattice we refer to \cite{HH1}.

\section*{Main Results}

Let $L$ be a finite distributive lattice of rank $d+1$ where $d$ is a positive integer, and $K[L]$  the polynomial ring over a field $K$. Let $I_L$ be the join-meet ideal of $L$ and $R(L)=K[L]/I_L.$ 

Throughout this paper we assume that the lattice $L$ is {\em simple}, that is, it has  no cut edge. By a {\em cut edge} of $L$ we mean  a pair $(a,b)$ of elements of $L$ with $\rank(b) = \rank(a)+1$ such that
$$|\{c\in L\:\; \rank(c)=\rank(a)\}|= |\{c\in L\:\; \rank(c)=\rank(b)\}|=1.$$
In particular, a simple distributive lattice of rank $d+1$ has at least two elements of rank $1$ and at least two elements of rank $d$.

There is no loss of generality in making this assumption. Let us suppose that $L$ has a cut edge $(a,b).$ Then it is clear that $I_L=I_{L_1}+I_{L_2}$ where $L_1$ is the sublattice 
of $L$ consisting of all  elements $c\in L$ such that $c\leq a,$ and $L_2$ is the sublattice of $L$ consisting of all  elements $c\in L$ such that $c\geq 
b.$  Since $I_{L_1}$ and $I_{L_2}$ are ideals generated by binomials in disjoint sets of variables, we get $R(L)=R(L_1)\otimes R(L_2)$ which implies that 
$\reg R(L)=\reg R(L_1)+\reg R(L_2).$

 By Theorem 10.1.3 in \cite{HHBook}, 
we know that the generators of $I_L$ form a Gr\"obner basis of $I_L$ with respect to the reverse lexicographic order on $K[L]$. Consequently, the initial ideal 
of $I_L$ is generated by all the squarefree monomials $ab$ where $a,b\in L$ are incomparable elements. This implies that the Hilbert series $H_{R(L)}(t)$ of $R(L)$ coincides with the Hilbert 
series of the Stanley-Reisner ring $K[\Delta (L)]$ where $\Delta (L)$ is the order complex of $L$, that is, the simplicial complex whose facets are the maximal 
chains of $L$. In particular, $R(L)$ and $K[\Delta (L)]$ have the same $h$-vector $h_{R(L)}$. Since $R(L)$ is Cohen-Macaulay, we may choose in $R(L)$ a regular sequence of 
linear forms, 
$\ub=u_1,\ldots,u_{\dim R(L)}.$  Then $R(L)$ and $R(L)/\ub R(L)$ have the same $h$--vector. By \cite[Theorem 20.2]{P}, we have 
$\reg R(L)=\reg(R(L)/\ub R(L))$, and, since $\dim(R(L)/\ub R(L))=0,$ the regularity of $R(L)/\ub R(L)$ is given by the degree of its $h$--vector \cite[Exercise 20.18]{Ei}. Consequently, 
$\reg R(L)=\deg h_{R(L)}.$

 The coefficients of $h_{R(L)}=h_{K[\Delta(L)]}$ have a nice combinatorial interpretation which we are going to recall below.

Let $P$ be the subposet of $L$ of 
the join-irreducible elements. By Birkoff's Theorem, $L$ equals the distributive lattice $\MI(P)$ of all poset ideals of $P.$ If $|P|=d+1$ for some positive 
integer $d,$ then $\rank L=d+1$ and $\dim(R(L))=d+2.$






 By \cite{BGS} or \cite[Section 2]{RW},  we have
\begin{equation}\label{eq1}
h_{K[\Delta(L)]} (t)= \sum_{S \subset [d]} \beta (S) t^{|S|}
\end{equation}
where $\beta (S)$ is the number of the linear extensions of the poset $P$ whose descent set is $S$. We recall that if $\pi = (a_1, \ldots, a_{d+1})$ is a permutation of $[d+1]$, then the descent set of $\pi$ is defined by $\Dc (\pi) = \{i : \;  a_i > a_{i+1}\}$. 

By \cite[Section 2]{BGS}, the number $\beta (S)$ may be also interpreted as follows.
 Let $\lambda$ be an edge-labeling of $L$. This means that each edge $x \rightarrow y$ in the Hasse diagram of $L$ has a label
$\lambda (x \rightarrow y)$. Here $x \rightarrow y$ means that $y$ covers $x$ in $L.$ Then each chain in $L$, say $x_0 \rightarrow x_1 \rightarrow x_2 \rightarrow\cdots \rightarrow x_k,$ is  labeled
by the $k$-tuple $(\lambda(x_0 \rightarrow x_1), \ldots, \lambda (x_{k-1} \rightarrow x_k))$. We compare two such $k$-tuples, say
$(a_1, \ldots, a_k)$ and $(b_1, \ldots, b_k)$, lexicographically, that is, $(a_1, \ldots, a_k) >_{\lex} (b_1, \ldots, b_k)$ if the most-left nonzero component 
of the vector $(a_1-b_1,\ldots,a_k-b_k)$ is positive. 

\begin{Definition}[\cite{BGS}]\label{edgelab}
The edge-labeling $\lambda$ of $L$ is called an \textit{EL-labeling} if for every interval $[x,y]$ in $L$:
\begin{enumerate}
\item[{\em (i)}] there is a unique chain $\cc: x= x_0 \rightarrow x_1 \rightarrow \cdots \rightarrow x_k =y$ such that
$\lambda(x_0 \rightarrow x_1)\leq \lambda(x_1 \rightarrow x_2) \leq \ldots \leq \lambda (x_{k-1} \rightarrow x_k)$;
\item[{\em (ii)}] for every other chain $\bb: x=y_0 \rightarrow y_1 \rightarrow \cdots \rightarrow y_k=y$ we have $\lambda (\bb) >_{\lex} \lambda(\cc)$.
\end{enumerate}
\end{Definition}



For a maximal chain $\cc: \min L= x_0 \rightarrow x_1 \rightarrow \cdots \rightarrow x_{d+1}= \max L$ in $L$, we define the {\em descent set} $\Dc(\cc)=\{i \in [d] : \lambda (x_{i-1} \rightarrow x_i) > \lambda (x_i \rightarrow x_{i+1})\}$.

We recall now Theorem 2.2 in \cite{BGS}.

\begin{Theorem}\label{bgs}\cite{BGS} Let $L$ be a graded poset of rank $d+1.$
For $S \subset [d]$, $\beta(S)$ equals the number of maximal chains $\cc$ in $L$ such that $\Dc(\cc) = S$. 
\end{Theorem}

\subsection*{Planar distributive lattices}

Let $\mathbb{N}^2$ be the infinite distributive lattice of all the pairs $(i,j)$ where $i,j$ are nonnegative integers. The partial order is defined as 
$(i,j)\leq (k,\ell)$ if $i\leq k$ and $j\leq \ell.$ A {\em planar distributive lattice} is a finite sublattice $L$ of $\mathbb{N}^2$ with $(0,0)\in L$ which has the following property: for any $(i,j), (k,\ell)\in L$ there exists a chain $\cc$ in $L$ of the form $\cc: x_0<x_1<\cdots <x_t$ with $x_s=(i_s,j_s)$ for 
$0\leq s\leq t,$ $(i_0,j_0)=(i,j)$, and $(i_t,j_t)=(k,\ell)$, such that $i_{s+1}+j_{s+1}=i_s+j_s+1$ for all $s.$

In the planar case, we may compute the regularity of $R(L)$ in terms of the cyclic sublattices of $L$. A sublattice of $L$ is called {\em cyclic} if it looks like in  Figure~\ref{cyclic} with some possible cut edges in between the squares. By a {\em square} in $L$ we mean a sublattice with elements $a,b,c,d$ such that 
$a\to b\to d$, $a\to c\to d$, and $b,c$ are incomparable.

\begin{figure}[hbt]
\begin{center}
\psset{unit=0.6cm}
\begin{pspicture}(1,-2)(5,5)
\rput(0,-1){
\rput(0,1){\pspolygon(2,2)(3,3)(4,2)(3,1)
\rput(2,2){$\bullet$}
\rput(3,3){$\bullet$}
\rput(4,2){$\bullet$}
\rput(3,1){$\bullet$}
}
\rput(0,3){\pspolygon(2,2)(3,3)(4,2)(3,1)
\rput(2,2){$\bullet$}
\rput(3,3){$\bullet$}
\rput(4,2){$\bullet$}
\rput(3,1){$\bullet$}
}
\psline(3,2)(3,1)
\rput(0,0){
\pspolygon(3,1)(2,0)(3,-1)(4,0)
\rput(3,1){$\bullet$}
\rput(2,0){$\bullet$}
\rput(3,-1){$\bullet$}
\rput(4,0){$\bullet$}
}
}
\end{pspicture}
\end{center}
\caption{Cyclic sublattice}
\label{cyclic}
\end{figure}
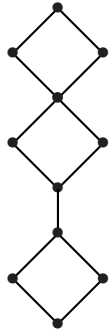

\begin{Lemma}\label{regcyclic}
Let $C$ be a cyclic lattice with $r$ squares. Then $\reg R(C)=r.$
\end{Lemma}

\begin{proof}
$I_C$ is generated by a regular sequence of length $r$ since $\ini_{\rev}(I_C)$ is generated by a regular sequence of monomials. Therefore, the Koszul complex of the generators of $I_C$ is the minimal free resolution of $R(C)$ over $K[C]$ and, hence, $\reg R(C)=r.$ 
\end{proof}

\begin{Theorem}\label{planar}
Let $L$ be a planar distributive lattice. Then $\reg R(L)$ equals the maximal number of squares in  a cyclic sublattice of $L.$
\end{Theorem}

In order to prove this theorem, we need some preparatory results. 

Let $L$ be a simple planar distributive lattice of rank $d+1$.  Let $\cc_0: x_0<x_1<\cdots<x_d< x_{d+1}$ be the chain of $L$ with $x_t=(i_t,j_t)$ for all 
$0\leq t\leq d+1$ and $(i_0,j_0)=(0,0), (i_{d+1},j_{d+1})=\max L$, having the following property: for any $(k,\ell)\in L$ with $k=i_t$ for some $t,$ we have $\ell\leq j_t.$ In other words, $\cc_0$ is the "most upper" chain of $L.$ We label the edges of $\cc_0$ by $\lambda(x_t\to x_{t+1})=t+1$ for $0\leq t\leq d.$ Next, we label all the edges in the Hasse diagram of $L$ as follows. If $i_{t+1}=i_t+1,$ in other words $x_t\to x_{t+1}$ is an horizontal edge, then we label by $t+1$ all the edges of $L$ of the form $(i_t,j)\to (i_{t+1},j)$. If $j_{t+1}=j_t+1,$ that is, $x_t\to x_{t+1}$ is a vertical edge, then we label by $t+1$ all the edges of $L$ of 
the form $(i,j_t)\to (i,j_{t+1})$.

\begin{Lemma}\label{unique}
Let $\cc: \min L=y_0<y_1<\cdots<y_{d+1}=\max L$ be an arbitrary maximal chain in $L$, $\cc\neq \cc_0.$ Then:
\begin{itemize}
	\item [(i)]  $\lambda(\cc)>_{\lex}\lambda(\cc_0).$
	\item [(ii)] there exists $q$ such that $\lambda(y_{q-1}\to y_q)>\lambda(y_q\to y_{q+1}).$
\end{itemize}
\end{Lemma}

\begin{proof} (i)
Since $\cc\neq \cc_0,$ we may choose $s=\min\{t: x_t\neq y_t\}.$ Let $x_t=(i_t,j_t)$ and $y_t=(k_t,\ell_t)$ for all $t.$ Assume that $i_{s-1}=i_s.$ The case 
$j_s=j_{s-1}$ can be treated in a similar way. Since $x_s\neq y_s$, we must have $k_s=i_{s-1}+1.$ Let $r=\max\{t: t>s-1, i_t=i_{s-1}\}$. Then 
$\lambda(y_{s-1}\to y_s)=\lambda(x_r\to x_{r+1})>\lambda(x_{s-1}\to x_s),$ which implies that $\lambda(\cc)>_{\lex}\lambda(\cc_0).$

For proving (ii),  we consider again the case $i_{s-1}=i_s$ and keep the notation of (i). Let $q=\max\{t: t>s-1, \ell_t=\ell_{s-1}\}. $ Then we get
\[
\lambda(y_q\to y_{q+1})=\lambda(x_{s-1}\to x_s)< \lambda(x_r\to x_{r+1})=\lambda(y_{s-1}\to y_s)\leq\lambda(y_{q-1}\to y_q).
\] The case $j_s=j_{s-1}$ can be done similarly.

\end{proof}

\begin{Proposition}\label{labeling}
The above defined edge-labeling of $L$ is an EL-labeling.
\end{Proposition}

\begin{proof}
Let $[x,y]$ be an interval of $L.$ We first prove condition (i) in Definition~\ref{edgelab}. In the first step, we show that, starting with an arbitrary chain 
$\cc$ from $x$ to $y$, we may find a chain $\gamma$ whose successive edges are labeled in increasing order. This shows the existence of the chain in (i). In 
the second step  we show the uniqueness.

 For an arbitrary chain $\cc: x=x_0=(i_0,j_0)\to x_1=(i_1,j_1)\to\cdots \to x_k=(i_k,j_k)=y$, we say that $x_t$ is an 
{\em upper corner} of $\cc$ if $j_t=j_{t-1}+1$ and $i_{t+1}=i_t+1.$ Similarly, $x_t$ is a {\em lower corner} of $\cc$ if $i_t=i_{t-1}+1$ and $j_{t+1}=j_t+1.$ It 
is almost obvious that if $x_t$ is not a corner or is an upper corner, than $\lambda(x_{t-1}\to x_t)<\lambda(x_t\to x_{t+1}).$ Indeed, if $x_t$ is not a corner, then the edges $x_{t-1}\to x_t$ and $x_t\to x_{t+1}$ are both either horizontal or vertical and, by the chosen labeling, we get 
$\lambda(x_{t-1}\to x_t)<\lambda(x_t\to x_{t+1}).$ Let now $x_t$ be an upper corner. We look at the edges $(i_t,k)\to (i_{t+1},k)$ and $(\ell, j_{t-1})\to (\ell, j_t)$ in the chain $\cc_0$. By the choice of $\cc_0,$ we have $\ell\leq i_t$ and $k\geq j_t$ which implies that $(\ell,j_t)\leq (i_t,j_t)\leq (i_t,k)$. Consequently, we get 
\[
\lambda(x_{t-1}\to x_t)=\lambda((\ell,j_{t-1})\to (\ell,j_t))<\lambda((i_t,k)\to (i_{t+1},k))=\lambda(x_t \to x_{t+1}).
\]
Let now $x_t$ be a lower corner of $\cc$ with $\lambda(x_{t-1}\to x_t)> \lambda(x_t \to x_{t+1}).$ We will replace $x_t$ in $\cc$ by $x_t^\prime=(i_t^\prime,j_t^\prime)$
where $i_t^\prime=i_{t-1}$ and $j_t^\prime=j_{t+1}.$ Now we need to explain that the edges $x_{t-1}\to x_t^\prime$ and $x_t^\prime\to x_{t+1}$ do appear in 
the Hasse diagram of $L.$ Let $(i_{t-1},j)\to (i_t,j)$ and $(i,j_t)\to (i,j_{t+1})$ be the edges of $\cc_0$ with the same labels as $x_{t-1}\to x_t$ and 
$x_t\to x_{t+1}$, respectively. As $\lambda(x_{t-1}\to x_t)> \lambda(x_t \to x_{t+1}),$ by the choice of $\cc_0,$ we must have $i\leq i_{t-1}$ and 
$j_{t+1}\leq j.$ Hence $x_{t-1}\to x_t^\prime$ and $x_t^\prime\to x_{t+1}$ are edges in $L.$

Now we look at the chain $\cc^\prime$ obtained from $\cc$ by replacing $x_t$ with $x_t^\prime.$ If it still has a lower corner, say $y_t$, with $\lambda(y_{t-1}\to y_t)> \lambda(y_t \to y_{t+1}),$ we replace 
$y_t$ by $y_t^\prime$ as we have done before in the chain $\cc.$ In this way, after finitely many such successive replacements, we get a new chain, say $\gamma,$ from $x$ to $y,$ whose edges are labeled in increasing order. 

For uniqueness, we proceed as follows. By Lemma~\ref{unique}, $\cc_0$ is the unique maximal chain of $L$ with the property that its edges are labeled in increasing order. Let us assume that we have $\gamma_1$ and $\gamma_2$ chains from $x$ to $y$ whose edges are labeled in increasing order. We extend these two chains to maximal chains in $L$, say $\Gamma_1$ and $\Gamma_2$. By suitable replacements of "bad" lower corners in $\Gamma_1$ and $\Gamma_2$ we reach the same maximal chain $\cc_0$. But these replacements do not affect $\gamma_1$ and $\gamma_2$, which implies that $\gamma_1=\gamma_2.$

Condition (ii) in Definition~\ref{edgelab} may be checked as in the proof of Lemma~\ref{unique} (ii).
\end{proof}

\begin{proof}[Proof of Theorem~\ref{planar}] Let $L$ be endowed with the above defined edge labeling and assume that the maximum number of squares in a cyclic 
sublattice of $L$ is $r.$
By Theorem~\ref{bgs} and equation (\ref{eq1}), we have to show that 
\[
r=\max\{|S|: \text{ there exists a maximal chain } \cc \text{ in }L \text{ with }\Dc(\cc)=S\}.
\]
Let $\cc: \min L=x_0<x_1<\cdots<x_{d+1}=\max L$ be a maximal chain in $L$ with $\Dc(\cc)=\{i_1,\ldots,i_m\}$. This means that for every $1\leq j\leq m,$ we have
\[
\lambda(x_{i_j-1}\to x_{i_j})> \lambda(x_{i_j}\to x_{i_j+1}).
\]
As we have already seen in the proof of Proposition~\ref{labeling}, $x_{i_1},\ldots.x_{i_m}$ must be lower corners of $\cc$ for which there exists 
$x^\prime_{i_1},\ldots,x^\prime_{i_m}\in L$ such that, for every $1\leq j\leq m,$ $x_{i_{j}-1}\to x^\prime_{i_j}$ and $x^\prime_{i_j}\to x_{i_j+1}$ are edges 
in the Hasse diagram of $L.$ Therefore, we get a sublattice $L^\prime$ of $L$ whose elements are the vertices of $\cc$ together with $x^\prime_{i_1},\ldots,x^\prime_{i_m}$ which is a cyclic sublattice with $m$ squares. Consequently, it follows that 
\[
r\geq\max\{|S|: \text{ there exists a maximal chain } \cc \text{ in }L \text{ with }\Dc(\cc)=S\}.
\] 

For the other inequality, let $L^\prime$ be a cyclic sublattice of $L$ which contains $r$ squares; see Figure~\ref{forproof}.

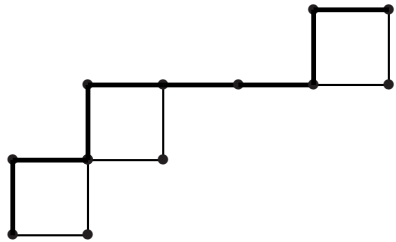
\begin{figure}[hbt]
\begin{center}
\psset{unit=1cm}
\begin{pspicture}(2,1)(5,4)

\pspolygon(1,1)(2,1)(2,2)(1,2)

\rput(1,1){$\bullet$}
\rput(2,1){$\bullet$}
\rput(2,2){$\bullet$}
\rput(1,2){$\bullet$}
\pspolygon(2,2)(3,2)(3,3)(2,3)
\rput(2,2){$\bullet$}
\rput(3,2){$\bullet$}
\rput(3,3){$\bullet$}
\rput(2,3){$\bullet$}
\psline[linewidth=1.8pt](3,3)(4,3) \rput(4,3){$\bullet$}
\psline[linewidth=1.8pt](4,3)(5,3)
\pspolygon(5,3)(6,3)(6,4)(5,4)
\rput(5,3){$\bullet$}
\rput(6,3){$\bullet$}
\rput(6,4){$\bullet$}
\rput(5,4){$\bullet$}
\psline[linewidth=1.8pt](1,2)(2,2)
\psline[linewidth=1.8pt](1,1)(1,2)
\psline[linewidth=1.8pt](2,2)(2,3)
\psline[linewidth=1.8pt](2,3)(3,3)
\psline[linewidth=1.8pt](5,3)(5,4)
\psline[linewidth=1.8pt](5,4)(6,4)
\end{pspicture}
\end{center}
\caption{The sublattice $L^\prime$}
\label{forproof}
\end{figure}

Let $\bb$ be the upper chain (drawn by the fat line in Figure~\ref{forproof}) in $L^\prime$ and $\cc$ the lower chain. Every lower corner in a square is a lower corner in $\cc$ which gives an element in the descent set $\Dc(\cc)$. Hence, 
\[
r\leq \Dc(\cc)\leq \max\{|S|: \text{ there exists a maximal chain } \cc \text{ in }L \text{ with }\Dc(\cc)=S\}.
\]
\end{proof}

\subsection*{Non-planar distributive lattices}
In the case of non-planar distributive lattices we give only  bounds for the regularity of the Hibi ring.

\begin{Lemma}\label{boole}
Let $B_n$ be the Boolean lattice of rank $n.$ Then $\reg R(B_n)=n-1.$
\end{Lemma}

\begin{proof}
Let $P=\{p_1,\ldots,p_n\}$ be the join-irreducible elements of $B_n.$ $P$ is an antichain, that is, $p_i$ is incomparable to $p_j$ for any $i\neq j.$ By using 
equation (\ref{eq1}), it follows that $\reg R(B_n)=\max\{|S|: \text{ there exists a linear extension of the poset } P \text{ whose descent set is  }S\}.$  As $P$ is an antichain, it follows that this maximum is $n-1,$ corresponding to the permutation $\pi$ of $P$ given by $\pi(p_i)=p_{n+1-i}$ for $1\leq i\leq n.$ Thus, 
$\reg R(B_n)=n-1.$
\end{proof}

\begin{Theorem}\label{non-planar}
Let $L=\Ic(P)$ be a non-planar distributive lattice. Then 
\small{
\[
|P|-1\geq\reg R(L)\geq \max\{|Q|: Q \text{ is a set of pairwise incomparable join-irreducible elements of }L\}-1.
\]}
\end{Theorem}

\begin{proof} The first inequality is trivially true since, by equation (\ref{eq1}), $\deg h_{R(L)}\leq |P|-1.$ Let us prove the second inequality.

Let $Q=\{p_1,\ldots,p_r\}$ be a maximal set of pairwise incomparable join-irreducible elements of $L.$ It follows that for any other join-irreducible element 
$p\in P$ we have either $p<p_i$ for some $i$ or $p>p_j$ for some $j.$ On the set $P$ of join-irreducible elements of $L$ we consider a new order, $\prec$, defined  as follows: $\prec$ is a linear order on the set $P^\prime=\{p\in P: p< p_i \text{ for some }i\}$ and on the set $P^{\prime\prime}=\{p\in P: p> p_j
\text{ for some }j\}$ which extends the original order on $P,$ that is, $p< q $  implies $p\prec q.$ Moreover, we set $\max_{\prec}P^\prime \prec p_i \prec 
\min_\prec P^{\prime\prime}$ for all $1\leq i \leq r.$ By the definition of $\prec$, it follows that, for any $p,q\in P,$ if $p\leq q,$ then $p\preceq q.$ 
By using \cite[Proposition 15.4]{S2}, we get $\beta_{(P,\leq)}(S)\geq \beta_{(P,\preceq)}(S)$ for any $S\subset[d].$ Together with equation (\ref{eq1}), this implies that 
\begin{equation}\label{eq2}
\reg R(L)=\deg h_{K[\Delta(L)]}\geq \deg h_{K[\Delta(L^\prime)]}=\reg R(L^\prime),
\end{equation}
where $L^\prime$ is the distributive lattice of the poset ideals of $(P,\preceq).$ It is obvious by  the definition of $\prec$ that the regularity of $R(L^\prime)$ is equal to the regularity of $R(B_r)$ where $B_r$ is the Boolean lattice of rank $r.$ Therefore, Lemma~\ref{boole} and inequality (\ref{eq2}) lead to the desired inequality.
\end{proof}

The next example shows that both inequalities in Theorem~\ref{non-planar} may be strict.

\begin{Example}{\em
Let $P=\{p_1,p_2,p_3,p_4,p_5\}$ be the poset with 
$
p_1<p_4, p_2<p_4, p_2<p_5,p_3<p_5
$
 and $L=\Ic(P).$ Then $\reg R(L)=3$ and the maximal number of pairwise incomparable   elements of $P$ is equal to $3.$
}
\end{Example}

As a corollary of the above theorems, we may characterize the distributive lattices $L$ with the property that the Hibi ring  $R(L)$ has a linear resolution over the polynomial ring $K[L]$.

\begin{Corollary}\label{linres}
Let $L$ be a distributive lattice. Then $R(L)$ has a linear resolution if and only if $L$ is the divisor lattice of $2\cdot 3^a$ for some $a\geq 0.$
\end{Corollary}

\begin{proof}
It is well known that if $L$ is the divisor lattice of $2\cdot 3^a$ for some $a\geq 0,$ then $R(L)$ has a linear resolution. Let now $L$ be a distributive lattice such that $R(L)$ has a linear resolution. If $L$ is non-planar, then it has at least three pairwise incomparable join-irreducible elements, thus 
$\reg R(L)\geq 2$, which is a contradiction to our hypothesis. Therefore, $L$ must be planar. In this case, the conclusion follows immediately as a consequence of  Theorem~\ref{planar}.
\end{proof}

{}

\end{document}